\documentclass[12pt, onesided, reqno]{amsart}

\usepackage{amssymb}
\usepackage{amsmath}
\usepackage{color}
\usepackage{enumerate}
\usepackage{graphicx}

\evensidemargin\oddsidemargin

\newtheorem{theorem}{Theorem}
\newtheorem{proposition}[theorem]{Proposition}
\newtheorem{corollary}[theorem]{Corollary}
\newtheorem{lemma}[theorem]{Lemma}

\theoremstyle{definition}

\newtheorem{remark}[theorem]{Remark}

\newcommand{\eqnsection}{
\renewcommand{\theequation}{\thesection.\arabic{equation}}
    \makeatletter
    \csname  @addtoreset\endcsname{equation}{section}
    \makeatother}
\eqnsection

\def\E{\mathbb{E}}

\def\N{\mathbb{N}}

\def\R{\mathbb{R}}

\def\Pb{\mathbb{P}}
\def\C{\mathbb{C}}

\def\F{\mathcal{F}}

\newcommand{\equi}{\mathop{\sim}\limits}
\def\={{\,\;\mathop{=}\limits^{\text{(law)}}\;\,}}

\def\qed{\hfill$\square$}

\newcommand*\pFqskip{8mu}
\catcode`,\active
\newcommand*\pFq{\begingroup
        \catcode`\,\active
        \def ,{\mskip\pFqskip\relax}%
        \dopFq
}
\catcode`\,12
\def\dopFq#1#2#3#4#5{%
        {}_{#1}F_{#2}\biggl[\genfrac..{0pt}{}{#3}{#4};#5\biggr]%
        \endgroup
}

\catcode`,\active

\catcode`\,12

\allowdisplaybreaks

\begin{document}

\title[]{The area under a spectrally positive stable excursion and other related processes}
\author[Christophe Profeta]{Christophe Profeta}

\address{
Universit\'e Paris-Saclay, CNRS, Univ Evry, Laboratoire de Math\'ematiques et Mod\'elisation d'Evry, 91037, Evry, France.
 {\em Email} : {\tt christophe.profeta@univ-evry.fr}
  }



\begin{abstract} We study the distribution of the area under the normalized excursion of a spectrally positive stable L\'evy process $L$, as well as the area under its meander, and under $L$ conditioned to stay positive. Our results involve a special case of Wright's function, which may be seen as a generalization of the classic Airy function appearing in similar Brownian's areas.
\end{abstract}

\maketitle

\section{Introduction}

Let $L$ be an $\alpha$-stable L\'evy process without negative jumps, with $\alpha>1$. We assume that $L$ is normalized to have characteristic exponent :
$$\ln\left(\E\left[e^{i \lambda L_1}\right]\right)= (i \lambda)^\alpha e^{-i\pi\alpha\, {\rm sgn}(\lambda)},\qquad \lambda \in \R.$$
It is well-known, see for instance \cite[Prop. 3.4.1]{SaTa}, that the distribution of the area under $L$ also follows a stable distribution :
\begin{equation}\label{eq:A1}
\int_0^1 L_u \,du \=  (1+\alpha)^{-\frac{1}{\alpha}} L_1. 
\end{equation}

The purpose of this note is to study the distribution of the area under three related processes : the normalized excursion of $L$, the meander of $L$, and $L$ conditioned to stay positive. These distributions have already been extensively studied in the Brownian case, see in particular the survey by Janson \cite{Jan} or the paper by Perman \& Wellner \cite{PeWe}. For all these Brownian areas, one observes the occurrence of the classic Airy function $\text{Ai}$. We shall prove that for spectrally positive stable L\'evy processes, the role of $\text{Ai}$ is played by the following M-Wright's function 
$$
\Phi_\alpha(x)= \frac{1}{\pi}\sum_{n=0}^{+\infty}  \frac{(-1)^{n}}{n!}\Gamma\left(\frac{1+n}{1+\alpha}\right)
\sin\left(\pi\frac{1+n}{\alpha+1}\right) (1+\alpha)^{\frac{n-\alpha}{1+\alpha}}   x^n.$$
The function $\Phi_\alpha$ is known to be related to a time-fractional diffusion equation, see \cite{MMP} and the references within. In particular, it admits the integral representation :
$$\Phi_{\alpha}(x) =\frac{1}{\pi} \int_0^{+\infty}e^{-\sin(\frac{\pi \alpha}{2}) \frac{z^{1+\alpha}}{1+\alpha}} \cos\left(\cos\left(\frac{\pi \alpha}{2}\right) \frac{z^{1+\alpha}}{1+\alpha} -zx\right) dz $$
from which we immediately see that $\Phi_{2}(x) = \text{Ai}(x)$. Another function of interest will be its  cosine counterpart :
$$\Psi_\alpha(x) = \frac{1}{\pi}\sum_{n=0}^{+\infty}  \frac{(-1)^{n}}{n!}\Gamma\left(\frac{1+n}{1+\alpha}\right)
\cos\left(\pi\frac{1+n}{\alpha+1}\right) (1+\alpha)^{\frac{n-\alpha}{1+\alpha}}   x^n.
$$

 \noindent
We now state the main results of this paper.

\subsection{The area under a normalized excursion}
Let $L^{\text{(ex)}}$ denote a normalized excursion of $L$ on the segment $[0,1]$ and 
set 
$$ \mathcal{A}_{\text{ex}} = \int_0^{1} L^{\text{(ex)}}_t dt.$$

\begin{theorem}\label{theo:1}
The double Laplace transform of $\mathcal{A}_{\emph{ex}}$ is given by :
$$\int_0^{+\infty}(1-e^{-\lambda t}) \E\left[e^{-t^{1+\frac{1}{\alpha}} \mathcal{A}_{\emph{ex}}}\right] t^{-(1+\frac{1}{\alpha})}   dt =\alpha\Gamma\left(1-\frac{1}{\alpha}\right) \left( \frac{\Phi_\alpha^\prime\left(0\right)}{\Phi_\alpha\left(0\right)}-\frac{\Phi_\alpha^\prime\left(\lambda\right)}{\Phi_\alpha\left(\lambda\right)}\right).  $$
\end{theorem}

\noindent
In the Brownian case, i.e. when $\alpha=2$, the distribution of $\mathcal{A}_{\text{ex}}$ is nowadays designated in the literature as the Airy distribution. We refer the reader to Louchard \cite{Lou1, Lou2} for a study of the Brownian excursion area via the Feynman-Kac formula, to Tak\'acs \cite{Tak} for an approach via random walks and to Flajolet \& Louchard \cite{FlLo} for a study of Airy distribution via its moments. \\

 \noindent
In the physics literature, this distribution has also appeared in the study of fluctuating interfaces, see  \cite{MaCo}. We finally mention that more recently, some authors have investigated the area under a normalized Bessel excursion and shown its relation with the cooling of atoms \cite{BAK, KMB}. \\

\noindent
As is the case for the Airy distribution, we may deduce from Theorem \ref{theo:1} a recurrence relation for the moments of $\mathcal{A}_{\text{ex}}$. To this end,  let us define a sequence $(B_{n,k}, 1\leq n,\, 1\leq k\leq n)$ by\footnote{For $x\in \R$ and $n\in \N$,  $(x)_n=x(x+1)\ldots(x+n-1)$ denotes the usual Pochammer symbol, with the convention $(x)_0=1$.
}
$$B_{n,1}= \frac{(2-\alpha)_{n-1}}{(n+1)(n+2)},\qquad n\geq 1,$$
and for $k\geq 1$,
\begin{equation}\label{eq:defBnk}
B_{n,k+1} = \frac{1}{k+1} \sum_{l=k}^{n-1}  \binom{n}{l} B_{n-l,1}\times B_{l, k},\qquad n\geq 1,
\end{equation}
This sequence corresponds to the values of the exponential  partial Bell polynomials taken on the sequence $(B_{n,1}, n\geq 1)$, see Comtet \cite[Section 3.3]{Com}.
We then set $c_0^{(\alpha)}=1$ and for $p\geq 1$, 
\begin{equation}\label{cp}
c_p^{(\alpha)} = \frac{1}{(2p)! \sqrt{\pi}}  \left(\frac{2}{\alpha}\right)^{p} \sum_{k=1}^{2p}  B_{2p,k}\,\Gamma\left(p+k+\frac{1}{2}\right) (2(\alpha-1))^{k}.
\end{equation} 
The coefficients $(c_p^{(\alpha)}, p\geq0)$ are the ones appearing in the asymptotic expansion of the function $\Phi_\alpha$, see Proposition \ref{prop:phi} in the Appendix.

\begin{corollary}
Let us set, for $n\geq 1$, 
$$\Omega_n = \frac{1}{n!}\frac{\E\left[\mathcal{A}^n_{\emph{ex}}\right]}{\alpha\Gamma\left(1-\frac{1}{\alpha}\right)}  \Gamma\left(\frac{(n-1)(\alpha+1)}{\alpha}+1\right).$$
Then, the sequence $(\Omega_n)$ follows the recurrence relation :
\begin{equation}\label{recu}
\Omega_n = c_{n-1}^{(\alpha)}\frac{(2n-1)(\alpha+1)-2}{2\alpha}-\sum_{k=1}^{n-1}  \Omega_kc_{n-k}^{(\alpha)}.
\end{equation}
As consequence,  we have the asymptotics\footnote{We write $a_p\asymp b_p$ as $p\rightarrow +\infty$ to state that their exist two constants $0<\kappa_1\leq \kappa_2<+\infty$ such that $\kappa_1 a_p \leq b_p \leq \kappa_2 a_p$ for $p$ large enough.} :
$$\left(\E[ \mathcal{A}_{\emph{ex}}^n]\right)^{\frac{1}{n}} \mathop{\asymp}\limits_{n\rightarrow +\infty}  n^{1-\frac{1}{\alpha}} \qquad \text{ and }\qquad \ln \Pb( \mathcal{A}_{\emph{ex}} >x)  \mathop{\asymp}\limits_{x\rightarrow +\infty} -x^{\frac{\alpha}{\alpha-1}}.  $$
\end{corollary}

\noindent
The first moments are given by :
$$\E[\mathcal{A}_{\text{ex}}]= \frac{\alpha-1}{2}\Gamma\left(1-\frac{1}{\alpha}\right)\quad \text{ and }\quad \E[\mathcal{A}^2_{\text{ex}}]=\frac{\Gamma\left(1-\frac{1}{\alpha}\right)(\alpha-1)\left(2\alpha+1\right) }{12 \Gamma\left(1+\frac{1}{\alpha}\right)}.
$$
Note that in the Brownian case, since $B_{n,1}=0$ as soon as $n\geq2$, the coefficients $c_n^{(2)}$ simplify to
$$c_{n}^{(2)}  =  \frac{1}{(2n)!\sqrt{\pi}} \left(\frac{1}{3}\right)^{2n} \Gamma\left(3n+\frac{1}{2}\right)
$$
and we recover a classic recurrence formula for the moments of Airy distribution, see \cite{Lou2}. In this case, a complete asymptotic expansion for $\Pb(\mathcal{A}_{\text{ex}}>x)$ was computed by Janson \& Louchard in \cite{JaLo}. \\

\subsection{The area under a stable meander} Let $L^{\text{(me)}}$ denotes the meander of $L$ on the segment $[0,1]$ and set 
$$ \mathcal{A}_{\text{me}} = \int_0^{1} L^{\text{(me)}}_t dt.$$

\begin{theorem}\label{theo:2}
The double Laplace transform of $\mathcal{A}_{\emph{me}}$ is given by :
\begin{equation}\label{eq:meander}
\int_0^{+\infty}e^{-\lambda t} \E\left[e^{-t^{1+\frac{1}{\alpha}} \mathcal{A}_{\emph{me}}}\right] t^{-\frac{1}{\alpha}}   dt =\pi \Gamma\left(1-\frac{1}{\alpha}\right)\frac{\Psi_\alpha^\prime(\lambda) \Phi_\alpha(\lambda)- \Phi_\alpha^\prime(\lambda)\Psi_\alpha(\lambda)}{\Phi_\alpha(\lambda)}.\end{equation}
In particular, 
$$\E[ \mathcal{A}_{\emph{me}}]= \Gamma\left(1-\frac{1}{\alpha}\right)\frac{\alpha+1}{2\alpha}$$
and there is the asymptotics, for $1<\alpha<2$ :
$$\Pb( \mathcal{A}_{\emph{me}}>x) \equi_{x\rightarrow+\infty} \frac{(\alpha-1)\Gamma(1+\alpha)\Gamma\left(1-\frac{1}{\alpha}\right)}{\Gamma(2-\alpha)\Gamma\left(2+\alpha-\frac{1}{\alpha}\right)} x^{-\alpha}. $$
\end{theorem}

In the Brownian case, this distribution was studied for instance by Tak\'acs \cite{Tak2}, see also Perman \& Wellner \cite{PeWe}.  Note that here, since the process is not pinned down at $t=1$, the presence of positive jumps yields a polynomial decay of the tail of $\mathcal{A}_{\text{me}}$, which is different from the case $\alpha=2$, see Janson \& Louchard \cite{JaLo}.

\begin{remark} Again, when $\alpha=2$, Theorem \ref{theo:2} simplifies. Indeed, observe that in this case, $\Psi_2$ is a solution of the following differential equation 
$$\Psi_2^{\prime\prime}(\lambda) = \lambda \Psi_2(\lambda) -\frac{1}{\pi}$$
hence the derivative of the numerator on the right-hand side of (\ref{eq:meander}) equals
$$\bigg(\Psi_2^\prime(\lambda) \text{Ai}(\lambda)- \text{Ai}^\prime(\lambda)\Psi_2(\lambda)\bigg)^\prime = 
- \frac{1}{\pi} \text{Ai}(\lambda).
$$
By integration, this implies that
$$\Psi_2^\prime(\lambda) \text{Ai}(\lambda)- \text{Ai}^\prime(\lambda)\Psi_2(\lambda) =\frac{1}{3\pi}-\frac{1}{\pi} \int_0^{\lambda} \text{Ai}(x) dx=\frac{1}{\pi}\int_\lambda^{+\infty} \text{Ai}(x) dx$$
which agrees for instance with \cite[Corollary 3.2]{PeWe}.
\end{remark}

\subsection{The area under $L$ conditioned to stay positive}
Let  $L^\uparrow$ be the process $L$ started from 0 and conditioned to stay positive. We set 
$$ \mathcal{A}^\uparrow = \int_0^{1} L^{\uparrow}_t dt.$$

\begin{theorem}\label{theo:3}
The double Laplace transform of $\mathcal{A}^\uparrow$ is given by :
\begin{equation}\label{eq:cond}
\int_0^{+\infty}e^{-\lambda t }\E\left[e^{-  t^{1+\frac{1}{\alpha}}\mathcal{A}^\uparrow }\right] dt 
= \pi \lambda\frac{ \Phi_\alpha^\prime\left(\lambda\right)\Psi_\alpha\left(\lambda\right) -\Psi_\alpha^\prime\left(\lambda\right)\Phi_\alpha\left(\lambda\right)}{\Phi_\alpha\left(\lambda\right)}- \frac{\Phi_\alpha^\prime\left(\lambda\right)}{\Phi_\alpha\left(\lambda\right)}.
\end{equation}
In particular, there is the asymptotics, for $1<\alpha<2$ :
$$\Pb( \mathcal{A}^\uparrow>x)\; \equi_{x\rightarrow+\infty}\;\frac{\Gamma(1+\alpha)}{\Gamma\left(1+\alpha-\frac{1}{\alpha}\right)\Gamma\left(2-\alpha\right)}x^{1-\alpha}.$$
\end{theorem}
\noindent
In the Brownian case, $\mathcal{A}^\uparrow$ corresponds, up to a $\sqrt{2}$ factor, to the integral on $[0,1]$ of a three-dimensional Bessel process started from 0.  As for the meander, Formula (\ref{eq:cond}) simplifies when $\alpha=2$, and the  right-hand side equals 
$$
- \lambda\frac{\int_\lambda^{+\infty} \text{Ai}(x) dx }{\text{Ai}\left(\lambda\right)}- \frac{\text{Ai}^\prime\left(\lambda\right)}{\text{Ai}\left(\lambda\right)}=\frac{-\lambda\int_\lambda^{+\infty} \text{Ai}(x) dx+ \int_\lambda^{+\infty} x \text{Ai}(x)dx}{\text{Ai}\left(\lambda\right)}=\frac{\int_0^{+\infty} x\text{Ai}(x+\lambda) dx}{\text{Ai}\left(\lambda\right)}
$$
which agrees with \cite[p.440]{BoSa}.

\subsection{Outline of the paper}

The remainder of the paper is divided as follows. Section 2 provides some notation as well as a key proposition which is of independent interest. Section 3, 4 and 5 give the proofs of the main Theorems, respectively on the normalized excursion, the meander, and $L$ conditioned to stay positive. Finally, Section 6 is an appendix on the Wright's function $\Phi_\alpha$, where we compute its asymptotic expansion at infinity.

\section{Preliminaries}

\subsection{Notations}
We start by recalling the definition of the considered processes, for which we mainly refer to Chaumont \cite{Cha}. We assume that $L$ is defined on the Skorokhod space of c\`adl\`ag processes. We denote by $\Pb_x$ its law when $L_0=x$, with the convention that $\Pb=\Pb_0$, and by $(\F_t, t\geq0)$ its natural filtration. 
Define, for $z\in \R$, 
$$T_z=\inf\{t\geq 0,\; L_t=z\}.$$
\begin{enumerate}
\item We denote by $\Pb_{x,y}^{\,t}$ the law of the bridge of $L$ of length $t$, going from $x$ to $y$.
\item We denote by $\Pb_x^{\uparrow}$ the law of $L$ started at $x>0$ and conditioned to stay positive. It is classically given by the $h$-transform :
\begin{equation}\label{abs+}
\forall \Lambda_t \in \F_t, \qquad \Pb_x^{\uparrow}(\Lambda_t) = \frac{1}{x}\E_x\left[L_t 1_{\Lambda_t}  1_{\{T_0>t\}}\right].
\end{equation}
This law admits a weak limit in the Skorokhod sense as $x\downarrow0$ which we shall denote $\Pb_0^{\uparrow}$.

\item We denote by $\Pb^{(me)}$ the law of the meander of $L$, which  is given by the limit
$$\forall \Lambda_t \in \F_t, \qquad \Pb^{(me)}(\Lambda_t) = \lim_{x\downarrow 0}\Pb_x\left(\Lambda_t |T_0>t\right).$$
\item Finally, we recall that, since $L$ is spectrally positive, the law of the stable excursion of length $t$ is equivalent to the bridge of $L$ conditioned to stay positive, and starting and ending at 0. From Lemma 4 in \cite{Cha}, this law may be for instance defined by 
\begin{equation}\label{def:ex}
\forall s<t\quad \Lambda_s \in \F_s, \qquad \Pb_{0,0}^{\uparrow \, t}(\Lambda_s) = \lim_{x\downarrow 0} \Pb_x\left(\Lambda_s|T_0=t\right).
\end{equation}
\end{enumerate}

\subsection{The key proposition}

The proofs of  Theorems \ref{theo:1}, \ref{theo:2} and \ref{theo:3} will rely heavily on the following proposition, which gives the  joint Laplace transform of the pair $\left(T_0, \,\int_0^{T_0} L_s ds\right)$. 
\begin{proposition}\label{prop:key}
For $z,\lambda \geq0$ and $\mu>0$ we have :
$$ \E_{z}\left[e^{-\lambda T_0- \mu \int_0^{T_0} L_s ds}\right]   =   \frac{\Phi_\alpha\left(\mu^{\frac{1}{1+\alpha}}  \left(z+\frac{\lambda}{\mu}\right)\right)}{\Phi_\alpha\left(\mu^{-\frac{\alpha}{1+\alpha}}\lambda\right)}.$$
\end{proposition}
\begin{proof}
We first assume that $\lambda=0$.  The law of $\int_0^{T_0} L_s ds$ has been studied by Letemplier \& Simon in \cite{LeSi}. In particular, they obtain the Mellin transform :
$$\E_1\left[\left(\int_0^{T_0}L_s ds\right)^\nu\right] = (\alpha+1)^\nu \frac{\Gamma(\frac{\alpha}{\alpha+1}) \Gamma(1-(\alpha+1)\nu)}{\Gamma(\frac{\alpha}{\alpha+1}-\nu) \Gamma(1-\nu)},\qquad\quad \nu < \frac{1}{1+\alpha}.$$
Replacing $\nu$ by $-\nu$ and using the definition of the Gamma function, we obtain :
$$
\int_0^{+\infty} \mu^{\nu-1} \E_1\left[e^{-\mu \int_0^{T_0} L_s ds} \right]d\mu =  (\alpha+1)^\nu\frac{ \Gamma(\nu)\Gamma(\frac{\alpha}{\alpha+1})\Gamma(1+(\alpha+1)\nu)}{\Gamma(\frac{\alpha}{\alpha+1}+\nu)\Gamma(\nu+1)}. 
$$
We now invert this Mellin transform following Janson \cite{JanMel} :
$$ \E_1\left[e^{-\mu \int_0^{T_0} L_s ds} \right] = 1+\Gamma\left(\frac{\alpha}{1+\alpha}\right) \sum_{n=1}^{+\infty}  \frac{(-1)^{n}}{n!}\frac{(1+\alpha)^{\frac{n}{1+\alpha}}}{ \Gamma\left(\frac{\alpha-n}{\alpha+1}\right)} \mu^{\frac{n}{1+\alpha}}$$
hence, by scaling, we thus obtain 
\begin{equation}\label{eq:I0}
\E_z\left[e^{-\mu\int_0^{T_0} L_s ds}\right] = \Gamma\left(\frac{\alpha}{1+\alpha}\right) (1+\alpha)^{\frac{\alpha}{1+\alpha}}  \, \Phi_\alpha\left(z\mu^{\frac{1}{1+\alpha}}\right)=\frac{\Phi_\alpha\left(z\mu^{\frac{1}{1+\alpha}}\right)}{\Phi_\alpha\left(0\right)}
\end{equation}
which is Proposition \ref{prop:key} when $\lambda=0$.
We now deal with the general case.
Let $x>y>0$. Applying the Markov property, we deduce from the absence of negative jumps that
\begin{equation}\label{prod}
\E_x\left[e^{-\mu \int_0^{T_0} L_s ds}\right] = \E_x\left[e^{-\mu \int_0^{T_y} L_s ds}\right] \times \E_y\left[e^{-\mu \int_0^{T_0} L_s ds}\right].
\end{equation}
Furthermore, by translation, we also have :
$$ \E_x\left[e^{-\mu \int_0^{T_y} L_s ds}\right]  = \E_{x-y}\left[e^{-\mu y T_0- \mu \int_0^{T_0} L_s ds}\right]. $$
Finally, setting  $x-y=z>0$ and $\lambda = \mu y>0$, and plugging (\ref{eq:I0}) in (\ref{prod}) 
yields 
$$ \E_{z}\left[e^{-\lambda T_0- \mu \int_0^{T_0} L_s ds}\right]   =  \frac{\E_x\left[e^{-\mu \int_0^{T_0} L_s ds}\right] }{\E_y\left[e^{-\mu \int_0^{T_0} L_s ds}\right] }=  \frac{\Phi_\alpha\left(\mu^{\frac{1}{1+\alpha}}  \left(z+\frac{\lambda}{\mu}\right)\right)}{\Phi_\alpha\left(\mu^{-\frac{\alpha}{1+\alpha}}\lambda\right)}$$
which ends the proof of Proposition \ref{prop:key}.
\end{proof}

\begin{remark}
Proposition \ref{prop:key} was proven in the Brownian case by Lefebvre \cite{Lef}, by applying the Feynman-Kac formula. A generalization where $T_0$ is replaced by the exit time from an interval was obtained by Lachal \cite{Lac} by a similar method.
\end{remark}

\section{The area under a spectrally positive stable excursion}

\subsection{Proof of Theorem \ref{theo:1} : the double Laplace transform} 

Notice first that by monotone convergence, the absolute continuity formula (\ref{def:ex}) remains true at $s=t$ :
\begin{equation}\label{limst}
\E_{0,0}^{\uparrow \, t}\left[e^{- \int_0^{t} L_u du}\right] = \lim_{x\downarrow 0} \E_{x}\left[e^{- \int_0^{t} L_u du}\big| T_0=t\right]. 
\end{equation}
Next, starting from Proposition \ref{prop:key}, we may write 
\begin{equation}\label{Phi-Phi}
\frac{\Phi_\alpha\left( z\right)}{\Phi_\alpha\left(0\right)}-\frac{\Phi_\alpha\left( z+\lambda\right)}{\Phi_\alpha\left(\lambda\right)}= \int_0^{+\infty}(1-e^{-\lambda t})\E_{z}\left[e^{- \int_0^{t} L_u du}\big| T_0=t\right] \Pb_z(T_0\in dt). 
\end{equation}
Recall now from Sato \cite[Theorem 46.3]{Sat} that since $L$ has no negative jumps, $T_0$ is a positive stable random variable of index $1/\alpha$, i.e.  for $z>0$ :
\begin{equation}\label{eq:PT0}
\Pb_z(T_0\in dt)/dt = \frac{1}{\pi} \sum_{n=1}^{+\infty} \frac{(-1)^{n-1}}{n!} \sin(n \pi /\alpha) \Gamma(1+n/\alpha)z^n t^{-1-n/\alpha}.
\end{equation}
Dividing (\ref{Phi-Phi}) by $z$ and letting $z\downarrow 0$, we deduce from (\ref{limst}) and (\ref{eq:PT0}) that
\begin{align*}
\frac{\Phi^\prime_\alpha\left( 0\right)}{\Phi_\alpha\left(0\right)}-\frac{\Phi^\prime_\alpha\left(\lambda\right)}{\Phi_\alpha\left(\lambda\right)}= \frac{\sin(\pi/\alpha)}{\pi}\Gamma\left(1+\frac{1}{\alpha}\right)  \int_0^{+\infty}(1-e^{-\lambda t})\E_{0,0}^{\uparrow \, t}\left[e^{- \int_0^{t} L_u du}\right] t^{-1-\frac{1}{\alpha}} dt.
\end{align*}
Theorem \ref{theo:1} now follows by the complement formula for the Gamma function and the scaling property. \qed \\

\subsection{Proof of Corollary 1 : study of the positive moments}
To get information on the moments, we shall work with a slight modification of the formula of Theorem \ref{theo:1}. Indeed, using the decomposition
\begin{multline*}
\alpha\Gamma\left(1-\frac{1}{\alpha}\right) \left( \frac{\Phi_\alpha^\prime\left(0\right)}{\Phi_\alpha\left(0\right)}-\frac{\Phi_\alpha^\prime\left(\lambda\right)}{\Phi_\alpha\left(\lambda\right)}\right)\\
=\int_0^{+\infty} \E\left[e^{-t^{1+\frac{1}{\alpha}} \mathcal{A}_{\text{ex}}}-1\right]   \frac{dt}{t^{1+\frac{1}{\alpha}}} 
+ \int_0^{+\infty} e^{-\lambda t}\E\left[1- e^{-t^{1+\frac{1}{\alpha}} \mathcal{A}_{\text{ex}}}\right] \frac{dt}{t^{1+\frac{1}{\alpha}}}  - \Gamma\left(-\frac{1}{\alpha}\right)\lambda^{1/\alpha}
\end{multline*}
we deduce the alternative formula 
\begin{equation}\label{ex:Exc2}
\int_0^{+\infty}e^{-\lambda t} \E\left[1-e^{-t^{1+\frac{1}{\alpha}} \mathcal{A}_{\text{ex}}}\right] t^{-(1+\frac{1}{\alpha})}   dt =\Gamma\left(-\frac{1}{\alpha}\right)\left(\frac{\Phi^\prime\left(\lambda\right)}{\Phi\left(\lambda\right)}+ \lambda^{1/\alpha} \right).  
\end{equation}
Setting $x=\lambda^{-(1+1/\alpha)}$ and making the change of variable $s=\lambda t$ then yields 
$$ \int_0^{+\infty}  e^{-s }\E\left[1-e^{- x s^{1+\frac{1}{\alpha}} \mathcal{A}_{\text{ex}}}\right] s^{-(1+\frac{1}{\alpha})}   ds =\Gamma\left(-\frac{1}{\alpha}\right)\left(x^{\frac{1}{1+\alpha}}  \frac{\Phi^\prime\left(x^{-\frac{\alpha}{1+\alpha}}\right)}{\Phi\left(x^{-\frac{\alpha}{1+\alpha}}\right)}+ 1 \right)  $$
and it remains  to compute the asymptotic expansion of both sides as $x\rightarrow0$. Using the asymptotic expansions of $\Phi_\alpha$ and $\Phi_\alpha^\prime$ given in Proposition \ref{prop:phi} in the Appendix, we get 
$$\sum_{n=1}^{+\infty} (-1)^{n} x^n \Omega_n \,\equi_{x\rightarrow 0}\, \frac{\sum_{p=0}^{+\infty} (-1)^{p+1} d_{p}^{(\alpha)} x^{p}  }{\sum_{p=0}^{+\infty} (-1)^p c_{p}^{(\alpha)} x^p }+1
$$
where we recall that $(c_p^{(\alpha)})$ is given by (\ref{cp}) and $(d_p^{(\alpha)})$ is given by 
$$d_0^{(\alpha)}=1\qquad \text{and}\qquad d_p^{(\alpha)} = c_{p}^{(\alpha)}-c_{p-1}^{(\alpha)}\left(\frac{(2p-1)(\alpha+1)-2}{2\alpha}\right)\quad\text{for }p\geq1.$$
By identification, we thus obtain  the announced recurrence relation :
$$\Omega_n = c_{n-1}^{(\alpha)}\frac{(2n-1)(\alpha+1)-2}{2\alpha} -\sum_{j=1}^{n-1}  \Omega_jc_{n-j}^{(\alpha)}.$$

\subsection{Proof of Corollary 1 : asymptotics}
To study the asymptotics of $(\Omega_n)$, observe first that by definition, this sequence is positive, and so is the sequence $(c_n^{(\alpha)})$. Therefore, we deduce from the recurrence relation (\ref{recu}) and from Corollary \ref{coro:c} in the Appendix that there exists a finite constant $\kappa_1$ such that for $n$ large enough
$$\Omega_n^{\frac{1}{n}} \leq \left( c_{n-1}^{(\alpha)}\frac{(2n-1)(\alpha+1)-2}{2\alpha}\right)^{\frac{1}{n}} \leq \kappa_1 \, n.$$
Stirling's formula now implies the upper bound for $\E[\mathcal{A}_{\text{ex}}^n]^{\frac{1}{n}}$, which may be written
\begin{equation}\label{Kas1}
\limsup_{n\rightarrow+\infty}  \frac{\E[\mathcal{A}_{\text{ex}}^n]^{\frac{1}{n}}}{n^{1-\frac{1}{\alpha}}}\leq \kappa_2
\end{equation}
for some finite constant $\kappa_2>0$. To obtain the tail decay of $\mathcal{A}_{\text{ex}}$, we shall rely on Kasahara's Tauberian theorem of exponential type. Indeed, applying \cite[Theorem 4]{Kas}, we deduce from (\ref{Kas1}), that there exists a constant $0<\kappa_3<+\infty$ such that 
$$\limsup_{x\rightarrow +\infty}  \frac{1}{x} \ln \Pb( \mathcal{A}_{\text{ex}} >x^{1-\frac{1}{\alpha}}) \leq -\kappa_3$$
which gives the announced upper bound.\\

\noindent
It does not seem easy to obtain a lower bound for $\Omega_n$  from the recurrence relation (\ref{recu}). Instead, we shall rather study directly the tail of the survival function of $\mathcal{A}_{\text{ex}}$. To do so, we recall the following absolute continuity formula for the normalized excursion, see \cite[Formula (11)]{Cha}:
$$\forall s<1,\quad \Lambda_s\in \F_s, \qquad \Pb_{0,0}^{\uparrow, 1}\left( \Lambda_s \right) =c \,\E_0^\uparrow \left[1_{\Lambda_s} \frac{j_{1-s}^\ast(L_s)}{L_s}  \right]$$
where $c$ is a normalization constant and, from Monrad \& Silverstein \cite[Formula (3.25)]{MoSi}, $j^\ast_{1-s}$ is a measurable function which admits the asymptotics :
$$ j^{\ast}_{1-s} (x) \equi_{x\rightarrow +\infty} \gamma (1-s)^{-1/\alpha} x^{\frac{\alpha}{2(\alpha-1)}} e^{- \eta (1-s)^{-\frac{1}{\alpha-1}} x^{\frac{\alpha}{\alpha-1}}}$$
where $\gamma$ and $\eta$ are two positive constants. Using this absolute continuity formula, we first deduce that :
\begin{align*}
\Pb(\mathcal{A}_{\text{ex}}>x) & \geq \Pb_{0,0}^{\uparrow, 1}\left( \int_{1/4}^{3/4}L_u^\uparrow du >x  \right)\\
&=c \,\E_0^{\uparrow}\left[\frac{j^\ast_{1/4}(L_{3/4}^\uparrow)}{L_{3/4}^\uparrow} 1_{\{\int_{1/4}^{3/4}L_u^\uparrow du >x\}}\right]\geq c\, \E_0^{\uparrow}\left[\frac{j^\ast_{1/4}(L_{3/4}^\uparrow)}{L_{3/4}^\uparrow} 1_{\left\{\inf\limits_{\frac{1}{4}\leq u \leq  \frac{3}{4}} L_u^\uparrow  >2 x\right\}}\right].
\end{align*}
Applying the Markov property, we further obtain 
\begin{align*}
\Pb(\mathcal{A}_{\text{ex}}>x) &\geq c\, \E_0^{\uparrow}\left[ \E_{L_{1/4}^\uparrow}^{\uparrow}\left[\frac{j_{1/4}^\ast(L_{1/2}^\uparrow)}{L_{1/2}^\uparrow} 1_{\left\{\inf\limits_{0\leq u \leq\frac{1}{2}} L_u^\uparrow  >2 x\right\}}\right]  1_{\{L_{1/4}^\uparrow>2x\}}\right]\\
&=c\,\E_0^{\uparrow}\left[  \frac{1}{L_{1/4}^\uparrow} \E_{L_{1/4}^\uparrow}\left[j_{1/4}^\ast(L_{1/2})1_{\left\{\inf\limits_{0\leq u \leq\frac{1}{2}} L_u  >2 x\right\}}\right]  1_{\{L_{1/4}^\uparrow>2x\}}\right],
\end{align*}
where we used in the last equality the absolute continuity formula (\ref{abs+}) for $\Pb_{L_{1/4}}^{\uparrow}$.
Now, since $L_{1/2} \geq \inf\limits_{0\leq u \leq\frac{1}{2}} L_u >2x$, we deduce from the asymptotics of $j^\ast_{1/4}$ that for $x$ large enough there exist two positive constants $\widetilde{\gamma}$ and $\widetilde{\eta}$ such that 
\begin{align*}
\Pb(\mathcal{A}_{\text{ex}}>x)&\geq \widetilde{\gamma} e^{-\widetilde{\eta}\,x^{\frac{\alpha}{\alpha-1}}}  \E_0^{\uparrow}\left[  \frac{1}{L_{1/4}^\uparrow} \Pb_{L_{1/4}^\uparrow}\left(\inf\limits_{0\leq u \leq\frac{1}{2}} L_u  >2 x \right)  1_{\{L_{1/4}^\uparrow>2x\}}\right].\\
&\geq \widetilde{\gamma} e^{-\widetilde{\eta}\,x^{\frac{\alpha}{\alpha-1}}}  \E_0^{\uparrow}\left[  \frac{1}{L_{1/4}^\uparrow} \Pb_{L_{1/4}^\uparrow}\left(\inf\limits_{0\leq u \leq\frac{1}{2}} L_u  >2 x,\; L_{1/2}  <4x \right)  1_{\{4x>L_{1/4}^\uparrow>3x\}}\right]\\
& \geq\widetilde{\gamma} e^{-\widetilde{\eta}\,x^{\frac{\alpha}{\alpha-1}}}  \frac{1}{4x}   \E_0^{\uparrow}\left[ \Pb_{0}\left(\inf\limits_{0\leq u \leq\frac{1}{2}} L_u  >-x,\;  L_{1/2}  <0\right)  1_{\{4x>L_{1/4}^\uparrow>3x\}}\right]\\
& =\widetilde{\gamma} e^{-\widetilde{\eta}\,x^{\frac{\alpha}{\alpha-1}}}  \  \frac{1}{4x}   \Pb_{0}\left(\inf\limits_{0\leq u \leq\frac{1}{2}} L_u  >-x,\;  L_{1/2}  <0\right)  \Pb_0^{\uparrow}\left(4x>L_{1/4}^\uparrow>3x\right)
\end{align*}
where, in the third inequality, we have used the independent increments property of $L$.
The lower bound now follows by taking the logarithm on each side, and using the limits 
$$\Pb_{0}\left(\inf\limits_{0\leq u \leq\frac{1}{2}} L_u  >-x,\;  L_{1/2}  <0\right)  \xrightarrow[x\rightarrow +\infty]{} \Pb_0\left(L_{1/2}  <0\right) = \frac{1}{\alpha}$$
and, see Chaumont \cite[p.12]{Cha},
$$ \Pb_0^{\uparrow}\left(4x>L_{1/4}^\uparrow>3x\right)\equi_{x\rightarrow +\infty} \kappa \,x^{-1-\alpha}$$
for some $\kappa>0$. The lower bound for $\E[\mathcal{A}_{\text{ex}}^{n}]^{\frac{1}{n}}$ follows then as before from Kasahara's Tauberian theorem of exponential type \cite[Theorem 4]{Kas}.\qed\\

\subsection{Study of some (fractional) negative moments}
As was observed by Flajolet \& Louchard \cite{FlLo} for the classic Airy distribution, one may also compute by recurrence some specific negative moments. These moments are related to the asymptotic expansion of $\Phi_\alpha$ at $\lambda=0$, which is given by  its very definition as a series.

 \begin{corollary}\label{coro:2}
Let us set for $n\geq1$ :
$$\Delta_n = \E\left[\mathcal{A}_{\emph{ex}}^{\frac{1-\alpha n}{\alpha+1}}\right] \Gamma\left(\frac{\alpha n-1}{\alpha+1}\right) (1+\alpha)^{-\frac{n+1}{1+\alpha}}. $$
Then, the sequence $(\Delta_n)$ follows the recurrence relation 
$$\Delta_n + \sum_{k=1}^{n-1} \binom{n}{k} \Delta_{n-k} \frac{\Gamma\left(\frac{\alpha}{\alpha+1}\right)}{\Gamma\left(\frac{\alpha-k}{\alpha+1}\right)}
 = \left(1+\frac{1}{\alpha}\right)   \Gamma\left(-\frac{1}{\alpha}\right) \left(\frac{\Gamma\left(\frac{\alpha}{\alpha+1}\right)}{\Gamma\left(\frac{\alpha-1-n}{\alpha+1}\right)} - \frac{\Gamma^2\left(\frac{\alpha}{\alpha+1}\right)}{\Gamma\left(\frac{\alpha-1}{\alpha+1}\right)\Gamma\left(\frac{\alpha-n}{\alpha+1}\right)}\right).
$$
\end{corollary}

\noindent
In particular, the first value is given by :
$$\E\left[\mathcal{A}_{\text{ex}}^{\frac{1-\alpha}{\alpha+1}}\right] = (1+\alpha)^{\frac{3+\alpha}{1+\alpha}} \frac{\Gamma\left(-\frac{1}{\alpha}\right)}{\Gamma\left(\frac{\alpha-1}{\alpha+1}\right)}\left(\frac{\Gamma\left(\frac{\alpha}{\alpha+1}\right)}{\Gamma\left(\frac{\alpha-2}{\alpha+1}\right)} - \frac{\Gamma^2\left(\frac{\alpha}{\alpha+1}\right)}{\Gamma^2\left(\frac{\alpha-1}{\alpha+1}\right)}\right).$$

\begin{proof}

Starting from Theorem \ref{theo:1} and multiplying both sides by $\Phi_\alpha$, we obtain 
\begin{multline*}
 \frac{\alpha}{\alpha+1}   \Phi_\alpha(\lambda)\E\left[\sum_{n=1}^{+\infty} \frac{(-1)^{n+1}}{n!} \lambda^n \mathcal{A}_{\text{ex}}^{\frac{1-\alpha n}{\alpha+1}}\Gamma\left(\frac{\alpha n-1}{\alpha+1}\right) \right]  \\
 = \Gamma\left(-\frac{1}{\alpha}\right) \left(\Phi_\alpha^\prime\left(\lambda\right) + (1+\alpha)^{\frac{1}{1+\alpha}} \frac{\Gamma\left(\frac{\alpha}{\alpha+1}\right)}{\Gamma\left(\frac{\alpha-1}{\alpha+1}\right)} \Phi_\alpha(\lambda)\right).
 \end{multline*}
 The result now follows by computing the Cauchy product on the left-hand side, and by identifying both expansions.
\end{proof}
\section{The area under a spectrally positive stable meander}

\subsection{Proof of Theorem \ref{theo:2} : the double Laplace transform}
We shall work here with a Laplace-Fourier transform to avoid integrability problems. Applying the Markov property and using Proposition \ref{prop:key}, we have for $z>0$ :
\begin{align*}
&\int_0^{+\infty}e^{-\lambda t }\E_z\left[e^{- i  \int_0^t L_u du} 1_{\{T_0>t\}}\right] dt \\
&= \int_0^{+\infty}e^{-\lambda t }\E_z\left[e^{- i  \int_0^t L_u du} \right] dt -\int_0^{+\infty}e^{-\lambda t }\E_z\left[e^{- i  \int_0^t L_u du} 1_{\{T_0\leq t\}}\right] dt \\
&= \int_0^{+\infty}e^{-\lambda t }\E_z\left[e^{-i  \int_0^t L_u du} \right] dt -\E_z\left[e^{-\lambda T_0-i  \int_0^{T_0} L_u du}\right]\int_0^{+\infty}e^{-\lambda t } \E_0\left[e^{-i  \int_0^t L_u du} \right]dt\\
&= \int_0^{+\infty}e^{-\lambda t }\E_0\left[e^{-i z t -i  \int_0^t L_u du} \right] dt  -  \frac{\Phi_\alpha\left(i^{\frac{1}{1+\alpha}}  \left(z+\frac{\lambda}{i}\right)\right)}{\Phi_\alpha\left(i^{-\frac{\alpha}{1+\alpha}}\lambda\right)}\int_0^{+\infty}e^{-\lambda t } \E_0\left[e^{-i  \int_0^t L_u du} \right]dt.
\end{align*}
We now divide this equality by $z$ and let $z\downarrow 0$. Using (\ref{eq:PT0}), the scaling property and the definition of the meander, the left-hand side converges towards 
$$\frac{1}{z}\int_0^{+\infty}e^{-\lambda t }\E_z\left[e^{- i  \int_0^t L_u du} \big|T_0>t\right] \Pb_z(T_0>t)dt \xrightarrow[z\downarrow 0]{}\frac{1}{\Gamma(1-\frac{1}{\alpha})}\int_0^{+\infty}e^{-\lambda t }\E\left[e^{- i  t^{1+1/\alpha} \mathcal{A}_{\text{me}}}\right] t^{-1/\alpha}  dt 
$$
while the right-hand side yields
\begin{equation}\label{eq:RHS}
-\int_0^{+\infty}i  t e^{-\lambda t }\E_0\left[e^{- i  \int_0^t L_u du} \right] dt - i^{\frac{1}{1+\alpha}} \frac{\Phi_\alpha^\prime\left(i^{-\frac{\alpha}{1+\alpha}}\lambda)\right)}{\Phi_\alpha\left(i^{-\frac{\alpha}{1+\alpha}}\lambda\right)}
\int_0^{+\infty}e^{-\lambda t } \E_0\left[e^{-i  \int_0^t L_u du} \right]dt.
\end{equation}
Let us set, to simplify the notation :
\begin{equation}\label{eq:F}
F_\alpha(\lambda)=\int_0^{+\infty}e^{-\lambda t } \E_0\left[e^{-i  \int_0^t L_u du} \right]dt. 
\end{equation}
From (\ref{eq:A1}), we deduce that for $\lambda \in \mathbb{C}$ :
\begin{align}
\notag F_\alpha(\lambda)&=  \int_0^{+\infty}e^{-\lambda t } \exp\left(e^{\frac{i \pi }{2}\alpha} \frac{t^{1+\alpha}}{1+\alpha}\right)dt\\
\notag &=\frac{1}{1+\alpha}\sum_{n=0}^{+\infty}  \frac{(-1)^n}{n!} \lambda^n \Gamma\left(\frac{n+1}{\alpha+1}\right) e^{\frac{i \pi (2-\alpha)}{2} \frac{n+1}{\alpha+1}}(1+\alpha)^{\frac{n+1}{\alpha+1}}\\
\label{phi+psi}&=\pi i^{-\frac{\alpha}{1+\alpha}} \left(\Psi_\alpha\left( i^{-\frac{\alpha}{1+\alpha}}\lambda\right) + i \Phi_\alpha \left( i^{-\frac{\alpha}{1+\alpha}}\lambda\right) \right)
\end{align}
which shows that $\Phi_\alpha$ and $\Psi_\alpha$ are closely related to the integral of $L$. Notice also that
\begin{equation}\label{F'}
-\int_0^{+\infty}i te^{-\lambda t } \E_0\left[e^{-i  \int_0^t L_u du} \right]dt = i F_\alpha^\prime(\lambda).
\end{equation}
Plugging (\ref{phi+psi}) and (\ref{F'}) in (\ref{eq:RHS}) finally yields the formula 
\begin{multline*}
\int_0^{+\infty}e^{-\lambda t} \E\left[e^{-it^{1+\frac{1}{\alpha}} \mathcal{A}_{\text{me}}}\right] t^{-\frac{1}{\alpha}}   dt \\=\Gamma\left(1-\frac{1}{\alpha}\right)i^{\frac{1-\alpha}{1+\alpha}}  \frac{\Psi_\alpha^\prime( i^{-\frac{\alpha}{1+\alpha}}\lambda) \Phi_\alpha( i^{-\frac{\alpha}{1+\alpha}}\lambda)- \Phi_\alpha^\prime( i^{-\frac{\alpha}{1+\alpha}}\lambda)\Psi_\alpha( i^{-\frac{\alpha}{1+\alpha}}\lambda)}{\Phi_\alpha( i^{-\frac{\alpha}{1+\alpha}}\lambda)}
\end{multline*}
and the announced result follows by analytical continuation.

\subsection{Proof of Theorem \ref{theo:2} :  first moment and asymptotics}
To simplify the following computation, we set 
$$H_\alpha(\lambda) = \pi \frac{\Psi_\alpha^\prime(\lambda) \Phi_\alpha(\lambda)- \Phi_\alpha^\prime(\lambda)\Psi_\alpha(\lambda)}{\Phi_\alpha(\lambda)}.$$
Differentiating the formula 
$$\int_0^{+\infty}e^{- t} \E\left[e^{-\left(\frac{t}{\lambda}\right)^{1+\frac{1}{\alpha}} \mathcal{A}_{\text{me}}}\right] t^{-\frac{1}{\alpha}}   dt = \Gamma\left(1-\frac{1}{\alpha}\right)\lambda^{1-\frac{1}{\alpha}}H_\alpha(\lambda)$$
 we deduce that
$$
 \left(\alpha+1\right)\int_0^{+\infty}e^{- t} \E\left[\mathcal{A}_{\text{me}}\,  e^{-\left(\frac{t}{\lambda}\right)^{1+\frac{1}{\alpha}} \mathcal{A}_{\text{me}}}\right] t  \, dt\\
  =\Gamma\left(1-\frac{1}{\alpha}\right)\bigg(  \left(\alpha-1\right)  \lambda^{2}H_\alpha(\lambda) +  \alpha \lambda^{3}H^\prime_\alpha(\lambda)\bigg)
$$
and the value  $\E\left[\mathcal{A}_{\text{me}}\right]$ will follow by letting $\lambda\rightarrow +\infty$ and applying the monotone convergence theorem on the left-hand side. Since the asymptotics of $\Phi_{\alpha}$ and $\Phi_\alpha^\prime$ are given in Proposition \ref{prop:phi} in the Appendix, it only remains to study those of $\Psi_\alpha$.
To this end, we observe from (\ref{phi+psi}) that 
$$\Psi_\alpha(\lambda) = \frac{i^{\frac{\alpha}{1+\alpha}}}{\pi} F_\alpha(\lambda i^{\frac{\alpha}{1+\alpha}}) - i \Phi_\alpha(\lambda).$$
Applying Watson's lemma, we deduce  from the definition (\ref{eq:F}) of $F_\alpha$ that
$$F_\alpha(\lambda i^{\frac{\alpha}{1+\alpha}}) \equi_{\lambda\rightarrow+\infty}  i^{- \frac{\alpha}{1+\alpha}}\sum_{n=0}^{+\infty} \frac{1}{n!} \frac{\Gamma\left(1+(1+\alpha)n\right)}{(1+\alpha)^n}  \lambda^{-n(1+\alpha)-1}$$
which implies, since $\Phi_\alpha$ decreases exponentially fast, the asymptotic expansion :
$$\Psi_\alpha(\lambda)  \equi_{\lambda\rightarrow+\infty} \frac{1}{\pi}\sum_{n=0}^{+\infty} \frac{1}{n!} \frac{\Gamma\left(1+(1+\alpha)n\right)}{(1+\alpha)^n}  \lambda^{-n(1+\alpha)-1}.$$
Therefore, going back to the definition of $H_{\alpha}$, we deduce that
\begin{equation}\label{eq:asympH}
H_\alpha(\lambda)  \mathop{=}\limits_{\lambda\rightarrow+\infty} \left(\lambda^{\frac{1}{\alpha}-1}  
-\frac{\alpha+1}{2\alpha} \lambda^{-2} +\Gamma(1+\alpha)\lambda^{-2-\alpha+\frac{1}{\alpha}}  \right)+ \text{o}(\lambda^{-2-\alpha + \frac{1}{\alpha}}).
\end{equation}
As a consequence, we obtain the limit 
$$\lim_{\lambda\rightarrow +\infty}  \left(\alpha-1\right)  \lambda^{2}H_\alpha(\lambda) +  \alpha \lambda^{3}H^\prime_\alpha(\lambda)
= \frac{(\alpha+1)^2}{2\alpha}$$
from which we deduce the value of the first moment 
$$\E\left[\mathcal{A}_{\text{me}}\right] = \Gamma\left(1-\frac{1}{\alpha}\right)\frac{\alpha+1}{2\alpha}.$$

\noindent
Next, to compute the asymptotics of $\Pb(\mathcal{A}_{\text{me}}>x)$, we shall work with Mellin transforms, which is a convenient tool when dealing with stable processes. Using Formula (\ref{eq:meander}) and applying the Fubini-Tonelli theorem, we have for $\nu\in \left(0,1-\frac{1}{\alpha}\right)$,

\begin{equation}\label{eq:MH}
\Gamma\left(1-\frac{1}{\alpha}\right)\int_0^{+\infty} \lambda^{\nu-1} H_\alpha(\lambda)d\lambda
= \frac{\alpha\Gamma(\nu)}{1+\alpha} \Gamma\left(\frac{\alpha-1-\alpha \nu}{\alpha+1}\right) \E\left[\left(\mathcal{A}_{\text{me}}\right)^{\frac{1-\alpha+\alpha\nu}{1+\alpha}}\right]. 
\end{equation}
Integrating twice by parts to remove the singularities at $\nu=1-\frac{1}{\alpha}$ and $\nu=2$, we further obtain 
\begin{multline*}
\Gamma\left(1-\frac{1}{\alpha}\right)\int_0^{+\infty} \lambda^{\nu-2} \left((\alpha-1) 2\lambda H_\alpha(\lambda)  + (4\alpha-1)\lambda^2 H_\alpha^{\prime} + \alpha \lambda^3 H_\alpha^{\prime\prime}(\lambda)\right) d\lambda\\
=(2-\nu) \Gamma(\nu)\Gamma\left(\frac{\alpha (2- \nu)}{\alpha+1}\right) 
(1-\alpha+\alpha\nu) \int_0^{+\infty} \lambda^{\nu-2} \lambda^{\frac{1}{\alpha}} \Pb\left(\mathcal{A}_{\text{me}}> \lambda^{1+\frac{1}{\alpha}}\right)d\lambda
\end{multline*}
where, as $\lambda\rightarrow+\infty$, the integrand on the left-hand side is equivalent to
$$(\alpha-1) 2\lambda H_\alpha(\lambda)  + (4\alpha-1)\lambda^2 H_\alpha^{\prime} + \alpha \lambda^3 H_\alpha^{\prime\prime}(\lambda) \equi_{\lambda\rightarrow +\infty} \Gamma(1+\alpha)(1+\alpha)^2(\alpha-1)\lambda^{-1-\alpha+\frac{1}{\alpha}}.$$
The Mellin transform on the left-hand side thus admits a simple pole at $\nu = 2+\alpha - \frac{1}{\alpha}$, hence so does the Mellin transform on the right-hand side. Applying the converse mapping theorem, we deduce that, as $\lambda\rightarrow +\infty$, 
$$
\Gamma\left(1-\frac{1}{\alpha}\right)\Gamma(1+\alpha)\lambda^{-1-\alpha+\frac{1}{\alpha}}
\equi_{\lambda\rightarrow+\infty}  -\Gamma\left(2+\alpha - \frac{1}{\alpha}\right)\Gamma\left(1-\alpha\right) 
\lambda^{\frac{1}{\alpha}} \Pb\left(\mathcal{A}_{\text{me}}> \lambda^{1+\frac{1}{\alpha}}\right)
$$
which yields the announced asymptotics.

\section{The area under $L$ conditioned to stay positive.}

\subsection{Proof of Theorem \ref{theo:3} : the double Laplace transform}
We proceed as for the meander. Applying the Markov property, we first have 
\begin{align*}
&\int_0^{+\infty}e^{-\lambda t }\E_z\left[L_t e^{- i  \int_0^t L_u du} 1_{\{T_0>t\}}\right] dt \\
&= \int_0^{+\infty}e^{-\lambda t }\E_z\left[L_t e^{-i  \int_0^t L_u du} \right] dt - \frac{\Phi_\alpha\left(i^{\frac{1}{1+\alpha}}  \left(z+\frac{\lambda}{i}\right)\right)}{\Phi_\alpha\left(i^{-\frac{\alpha}{1+\alpha}}\lambda\right)}\int_0^{+\infty}e^{-\lambda t } \E_0\left[L_t e^{-i  \int_0^t L_u du} \right]dt.
\end{align*}
Dividing both sides by $z$ and letting $z\downarrow 0$  then yields by definition of $\Pb_0^{\uparrow}$ 
\begin{multline*}
\int_0^{+\infty}e^{-\lambda t }\E_0^\uparrow\left[e^{- i \int_0^t L_u du}\right]dt \\
=\int_0^{+\infty}e^{-\lambda t }\E_0\left[(1-it L_t) e^{-i  \int_0^t L_u du} \right] dt - i^{\frac{1}{1+\alpha}} \frac{\Phi_\alpha^\prime\left(i^{-\frac{\alpha}{1+\alpha}}\lambda\right)}{\Phi_\alpha\left(i^{-\frac{\alpha}{1+\alpha}}\lambda\right)}
\int_0^{+\infty}e^{-\lambda t } \E_0\left[L_t e^{-i  \int_0^t L_u du} \right]dt.
\end{multline*}
Observe next that, integrating by parts the definition (\ref{eq:F}) of $F_\alpha$, we also have 
$$F_\alpha(\lambda) = \frac{1}{\lambda} - \frac{i}{\lambda}\int_0^{+\infty}e^{-\lambda t } \E_0\left[L_t e^{-i  \int_0^t L_u du} \right]dt.$$
As a consequence, we deduce from (\ref{phi+psi}) that 
\begin{align*}
&\int_0^{+\infty}e^{-\lambda t }\E_0^\uparrow\left[e^{- i  \int_0^t L_u du}\right] dt \\
&=-\lambda F_\alpha^\prime(\lambda)  - i^{-\frac{\alpha}{1+\alpha}} \frac{\Phi_\alpha^\prime\left(i^{-\frac{\alpha}{1+\alpha}}\lambda\right)}{\Phi_\alpha\left(i^{-\frac{\alpha}{1+\alpha}}\lambda\right)}\left(1-\lambda F_\alpha(\lambda)\right)\\
&=\pi i^{-\frac{2\alpha}{1+\alpha}}\lambda\frac{\Phi_\alpha^\prime\left(i^{-\frac{\alpha}{1+\alpha}}\lambda\right)\Psi_\alpha\left(i^{-\frac{\alpha}{1+\alpha}}\lambda\right)  -  \Psi^\prime_\alpha\left(i^{-\frac{\alpha}{1+\alpha}}\lambda\right)\Phi\left(i^{-\frac{\alpha}{1+\alpha}}\lambda\right) }{\Phi\left(i^{-\frac{\alpha}{1+\alpha}}\lambda\right)} -  i^{-\frac{\alpha}{1+\alpha}}  \frac{\Phi^\prime\left(i^{-\frac{\alpha}{1+\alpha}}\lambda\right)}{\Phi\left(i^{-\frac{\alpha}{1+\alpha}}\lambda\right)}
\end{align*}
and the result  follows as before by analytic continuation.

\subsection{Proof of Theorem \ref{theo:3} : asymptotics}
The asymptotics of the tail of $\mathcal{A}^{\uparrow}$ is easy to obtain as we can work with Laplace transforms and use repeatedly Karamata's Tauberian theorem. Indeed,  since we are dealing with monotone integrands, there is the asymptotics  for $\alpha<2$ :
$$
\int_0^{+\infty}e^{-\lambda t }\E\left[1-e^{- t^{1+\frac{1}{\alpha}}\mathcal{A}^\uparrow }\right] dt =\frac{1}{\lambda}+\lambda H_\alpha(\lambda)+\frac{\Phi_\alpha^\prime\left(\lambda\right)}{\Phi_\alpha\left(\lambda\right)}\equi_{\lambda\rightarrow +\infty} \frac{\Gamma(1+\alpha)}{\lambda^{1+\alpha-\frac{1}{\alpha}}}
$$
which implies that
$$\E\left[1-e^{- t^{1+\frac{1}{\alpha}}\mathcal{A}^\uparrow }\right]\;\equi_{t\rightarrow 0}\; \frac{\Gamma(1+\alpha)}{\Gamma(1+\alpha-\frac{1}{\alpha})} t^{\alpha-\frac{1}{\alpha}}$$
which in turn implies that
$$
\Pb(\mathcal{A}^{\uparrow}>x) \;\equi_{x\rightarrow +\infty}\; \frac{\Gamma(1+\alpha)}{\Gamma\left(1+\alpha-\frac{1}{\alpha}\right)\Gamma\left(2-\alpha\right)}x^{1-\alpha}.
$$

\section{Appendix on M-Wright's functions}\label{sec:A}

We gather and prove in this section several useful formulae for the M-Wright's function $\Phi_\alpha$ and its derivative.

\subsection{The integral representation}
Following the notation of \cite{GLM}, the classic Wright's function is defined by :
$$\phi(\rho, \beta; z) = \sum_{k=0}^{+\infty} \frac{z^k}{k! \Gamma(\rho k+\beta)},\qquad \rho>-1,\; \beta\in \C.$$
A special case of this function is obtained when $\rho = \beta-1=-\frac{1}{1+\alpha}$, for which, applying the compensation formula for the Gamma function, we obtain :
\begin{align*}
\Phi_\alpha(x) &= \frac{1}{\pi} (1+\alpha)^{-\frac{\alpha}{1+\alpha}} \,\phi\left(-\frac{1}{1+\alpha}, \frac{\alpha}{1+\alpha}; -x(1+\alpha)^{\frac{1}{1+\alpha}}\right)\\
&= \frac{1}{\pi}\sum_{n=0}^{+\infty}  \frac{(-1)^{n}}{n!}\Gamma\left(\frac{1+n}{1+\alpha}\right)
\sin\left(\pi\frac{1+n}{\alpha+1}\right) (1+\alpha)^{\frac{n-\alpha}{1+\alpha}}   x^n
\end{align*}
which is nowadays referred in the literature as an M-Wright's function. This function is also closely related to the probability density of the positive stable distribution of parameter $\frac{1}{\alpha+1}$, see Sato \cite[p.88]{Sat}. To get an integral expression, recall the formulae, since $\alpha>1$ :
$$\int_0^{\infty e^{\pm i\pi/(1+\alpha)}} z^{n} e^{\frac{z^{1+\alpha}}{1+\alpha}} dz =e^{\pm i\frac{\pi(n+1)}{1+\alpha}} (1+\alpha)^{\frac{n-\alpha}{\alpha+1}} \Gamma\left(\frac{n+1}{\alpha+1}\right)$$
which implies
\begin{align*}
\Phi_\alpha(x) &=\frac{1}{2\pi i}  \sum_{n=0}^{+\infty} \frac{(-1)^n}{n!} x^n \left(\int_0^{\infty e^{i\pi/(1+\alpha)}} z^{n} e^{\frac{z^{1+\alpha}}{1+\alpha}} dz - \int_0^{\infty e^{- i\pi/(1+\alpha)}} z^{n} e^{\frac{z^{1+\alpha}}{1+\alpha}} dz\right)  \\
& = \frac{1}{2\pi i} \int_{\infty e^{- i\pi/(1+\alpha)}}^{\infty e^{i\pi/(1+\alpha)}} e^{\frac{z^{1+\alpha}}{1+\alpha}-zx} dz. 
\end{align*}
Applying Cauchy's integral theorem, we may deform the path of integration to pass by the imaginary axis, and thus obtain the integral representation :
\begin{equation}\label{eq:phiint}
\Phi_{\alpha}(x) = \frac{1}{\pi}\int_0^{+\infty}e^{-\sin(\frac{\pi \alpha}{2}) \frac{z^{1+\alpha}}{1+\alpha}} \cos\left(\cos\left(\frac{\pi \alpha}{2}\right) \frac{z^{1+\alpha}}{1+\alpha} -zx\right) dz. 
\end{equation}

\subsection{Asymptotic expansion of $\Phi_\alpha$ and $\Phi_\alpha^\prime$}
We now study the asymptotics of $\Phi_\alpha(x)$ as $x\rightarrow +\infty$. A general (theoretical) asymptotic expansion for $\phi$ was computed by Wright \cite{Wri} (see also \cite[Theorem 2.1.3]{GLM}), but it seems difficult to extract from his formula an explicit expression for the coefficients. This will be our objective here.\\

\noindent
Recall the following definition of the coefficients $(c_p^{(\alpha)})$ :
$$
c_0^{(\alpha)}=1 \qquad \text{and}\qquad c_p^{(\alpha)} = \frac{1}{(2p)! \sqrt{\pi}}  \left(\frac{2}{\alpha}\right)^{p} \sum_{k=1}^{2p}  B_{2p,k}\,\Gamma\left(p+k+\frac{1}{2}\right) (2(\alpha-1))^{k}, \quad p\geq1
$$
where the sequence $(B_{n,k})$ is defined for any $n\geq1$ by 
$$
B_{n,1}= \frac{(2-\alpha)_{n-1}}{(n+1)(n+2)}\qquad \text{and}\qquad 
B_{n,k+1} = \frac{1}{k+1} \sum_{l=k}^{n-1}  \binom{n}{l} B_{n-l,1}\times B_{l, k}\quad k\geq1.
$$
We start with a simple lemma.

\begin{lemma}\label{lem:Binfty}
For any $k\geq 1$, there is the upper bound :
\begin{equation}\label{rec}
\forall n\geq 1,\qquad \frac{B_{n,k}}{n!} \leq \frac{(B_\infty)^k}{k!}\leq \frac{(1/4)^k}{k!}\qquad \text{where }\quad B_\infty = \sum_{n=1}^{+\infty} \frac{B_{n,1}}{n!}.
\end{equation}
\end{lemma}

\begin{proof}
We prove the first inequality by iteration on $k$. For $k=1$, this is obviously true since $\dfrac{B_{n,1}}{n!} \leq B_\infty$. Assume now that (\ref{rec}) holds for all the integers up to some $k$. Then, from the definition (\ref{eq:defBnk}) of the sequence $(B_{n,k})$, we have
$$\frac{B_{n,{k+1}}}{n!} = \frac{1}{k+1} \sum_{l=k}^{n-1}  \frac{B_{n-l,1}}{(n-l)!} \frac{B_{l,k}}{l!}\leq  \frac{(B_{\infty})^k}{(k+1)!} \sum_{l=k}^{n-1}  \frac{B_{n-l,1}}{(n-l)!}\leq \frac{(B_{\infty})^{k+1}}{(k+1)!}$$
which proves the first inequality. The second one follows from the fact that $(2-\alpha)_{n-1} \leq  (n-1)!$ for all $n\geq 1$, hence, going back to the definition of the sequence $(B_{n,1})$, 
$$ B_\infty \leq  \sum_{n=1}^{+\infty} \frac{1}{n(n+1)(n+2)} = \frac{1}{4}.$$ 
\end{proof}
We may now compute the asymptotic expansion of $\Phi_\alpha$ and $\Phi_\alpha^\prime$.

\begin{proposition}\label{prop:phi}
We have the asymptotic expansions as $x\rightarrow +\infty$ :
$$\Phi_{\alpha}(x) \equi_{x\rightarrow +\infty} \frac{1}{\sqrt{2\pi\alpha}}x^{\frac{1-\alpha}{2\alpha}} e^{- \frac{\alpha}{\alpha+1} x^{1+1/\alpha}}\sum_{p=0}^{+\infty} 
(-1)^pc_{p}^{(\alpha)} x^{-p(1+1/\alpha)} $$
and
$$\Phi_\alpha^\prime(x) \equi_{x\rightarrow +\infty}  \frac{1}{\sqrt{2\pi\alpha}}x^{\frac{3-\alpha}{2\alpha}} e^{- \frac{\alpha}{\alpha+1} x^{1+1/\alpha}}\sum_{p=0}^{+\infty} 
(-1)^{p+1}d_p^{(\alpha)} x^{-p(1+1/\alpha)} $$
where $d_0^{(\alpha)}=1$ and for $p\geq 1$ :
$$d_p^{(\alpha)} = c_{p}^{(\alpha)}-c_{p-1}^{(\alpha)}\left(\frac{(2p-1)(\alpha+1)-2}{2\alpha}\right).$$
\end{proposition}

\begin{proof} Since these asymptotics are already known for the Airy function, we shall assume in the following that $\alpha<2$.
We start with the asymptotic expansion of $\Phi_\alpha$. Coming back to the formula (\ref{eq:phiint}), we may write, using the change of variable  $z=x^{\frac{1}{\alpha}}y$ :
$$\Phi_{\alpha}(x) = \frac{x^{\frac{1}{\alpha}}}{2\pi} \int_{-\infty}^{+\infty}e^{i x^{1+\frac{1}{\alpha}}\left(e^{i\frac{\pi\alpha}{2}} \frac{y^{1+\alpha}}{1+\alpha} -y\right)} dy.$$
Applying Cauchy's integral theorem, we first deform the path of integration to pass through the line $\{z\in \C, \, \Im(z)=-1\}$ :
$$\Phi_\alpha(x) = \frac{x^{\frac{1}{\alpha}}}{2\pi} \int_{-\infty}^{+\infty}e^{i x^{1+\frac{1}{\alpha}}\left(i^\alpha \frac{(u-i)^{1+\alpha}}{1+\alpha} -(u-i)\right)} du.$$ 
Recall next the following Taylor expansion :
$$\frac{(u-i)^{1+\alpha}}{1+\alpha}= \frac{(-i)^{1+\alpha}}{1+\alpha} +  u (-i)^\alpha+ \frac{u^2}{2} \alpha (-i)^{\alpha-1} + \alpha (\alpha-1)\frac{u^3}{2} \int_0^1 (1-t)^{2} (tu-i)^{\alpha-2}dt$$
as well as Euler's integral formula for the hypergeometric function $_2F_1$ :
$$ \int_0^1 (1-t)^{2} (tu-i)^{\alpha-2}dt = \frac{1}{3}(-i)^{\alpha-2}\pFq{2}{1}{2-\alpha\quad 1}{4}{-iu}.$$
Setting $\xi=x^{1+1/\alpha}$ and making the change of variable  $u \sqrt{\xi}=z$, we further obtain 
\begin{equation}\label{F+}
\Phi_\alpha(\xi^{\frac{\alpha}{\alpha+1}}) = \xi^{\frac{1}{1+\alpha}-\frac{1}{2}}e^{-\frac{\alpha}{\alpha+1}\xi}\int_{-\infty}^{+\infty} e^{-\alpha\frac{z^2}{2} +  z^2\alpha (\alpha-1) \varphi\left(\frac{z}{\sqrt{\xi}}\right)}dz
\end{equation}
where
$$\varphi\left(\frac{z}{\sqrt{\xi}}\right)=-\frac{i}{6} \frac{z}{\sqrt{\xi}}\,  \pFq{2}{1}{2-\alpha\quad 1}{4}{-i\frac{z}{\sqrt{\xi}}}.$$ 
For $|z|<\sqrt{\xi}$, the definition of $_2F_1$ as a series yields the alternative expression :
$$\varphi\left(\frac{z}{\sqrt{\xi}}\right)=  \sum_{n=1}^{+\infty} \frac{(2-\alpha)_{n-1}}{(2+n)!}  \left(-i\frac{z}{\sqrt{\xi}}\right)^n.$$
By definition of the sequence $(B_{n,k},\; 1\leq n,\, 1\leq k\leq n)$, we have, still for  $|z|<\sqrt{\xi}$ :
$$e^{ z^2\alpha (\alpha-1) \varphi\left(\frac{z}{\sqrt{\xi}}\right)}
=1+ \sum_{n=1}^{+\infty} \frac{1}{n!}
\left(\frac{-i z}{\sqrt{\xi}}\right)^n \sum_{k=1}^n \big(z^{2} \alpha(\alpha-1)\big)^k B_{n,k}$$
which may be read as an asymptotic expansion in $\xi$. It remains now to plug this expansion in (\ref{F+}) and integrate term by term to obtain the announced result. However, some care is needed as the convergence of the expansion is not uniform in $z$, hence we cannot apply directly \cite[Theorem 1.7.5]{BlHa} for instance. Let $N\in \N$. Using the integral definition of the Gamma function and only keeping even terms, we have
\begin{align*}
&\left| \xi^{\frac{1}{2}-\frac{1}{1+\alpha}}e^{\frac{\alpha}{\alpha+1}\xi} \Phi_\alpha(\xi^{\frac{\alpha}{\alpha+1}}) -\frac{1}{\sqrt{2\pi\alpha}}\sum_{p=0}^{N-1} 
(-1)^pc_{p}^{(\alpha)} \xi^{-p} 
  \right|\\
  &\qquad \leq\int_{-\infty}^{+\infty} e^{-\alpha\frac{z^2}{2}  }  \left|   e^{ z^2\alpha (\alpha-1) \varphi\left(\frac{z}{\sqrt{\xi}}\right)}
  -  \frac{1}{\sqrt{2\pi \alpha}}-\sum_{p=1}^{N-1} \frac{(-1)^p}{(2p)!}
\left(\frac{z}{\sqrt{\xi}}\right)^{2p} \sum_{k=1}^{2p} \big(z^{2} \alpha(\alpha-1)\big)^k B_{2p,k}  \right|dz
\end{align*}
and we need to prove that this last quantity is smaller than $K\,\xi^{-N}$ for some constant $K$ independent from $\xi$.
We decompose this last integral according as $|z|<r \sqrt{\xi}$ or $|z|\geq r \sqrt{\xi}$ where $0<r<1$ is fixed. On the one hand, when $|z|<r \sqrt{\xi}$, we deduce from Lemma \ref{lem:Binfty} :
\begin{align*}
&\int_{-r \sqrt{\xi}}^{r\sqrt{\xi}} e^{-\alpha\frac{z^2}{2}  }  \left|  \sum_{p=N}^{+\infty} \frac{(-1)^p}{(2p)!}
\left(\frac{z}{\sqrt{\xi}}\right)^{2p} \sum_{k=1}^{2p} \big(z^{2} \alpha(\alpha-1)\big)^k B_{2p,k}  \right|dz\\
&\qquad\leq\int_{-r \sqrt{\xi}}^{r\sqrt{\xi}} e^{-\alpha\frac{z^2}{2}  } \sum_{p=0}^{+\infty} \frac{1}{(2p+2N)!}
\left(\frac{z}{\sqrt{\xi}}\right)^{2p+2N} \sum_{k=1}^{2p+2N} \big(z^{2} \alpha(\alpha-1)\big)^k B_{2p+2N,k} dz\\
&\qquad\leq  \frac{1}{\xi^{N}} \int_{-r \sqrt{\xi}}^{r\sqrt{\xi}}  e^{-\alpha\frac{z^2}{2}  }  \sum_{p=0}^{+\infty} 
 z^{2N} r^{2p} \sum_{k=1}^{2p+2N} \big(z^{2} \alpha(\alpha-1)\big)^k \frac{(1/4)^k}{k!}  dz\\
&\qquad\leq  \frac{1}{\xi^{N}} \sum_{p=0}^{+\infty} 
r^{2p} \int_{-r \sqrt{\xi}}^{r\sqrt{\xi}}  z^{2N} e^{-\alpha\frac{z^2}{2}  } e^{\frac{z^2\alpha(\alpha-1)}{4}}  dz\\
&\qquad\leq  \frac{1}{\xi^{N}}\frac{1}{1-r^2}
 \int_{-\infty}^{+\infty}  z^{2N} e^{-\frac{z^2}{4}\alpha(3-\alpha)   } dz
\end{align*}
which is finite since $\alpha\in (1,2)$. On the other hand, when $|z|\geq r \sqrt{\xi}$, we simply use the triangular inequality and first write :
$$\int_{|z|\geq r\sqrt{\xi}} e^{-\alpha\frac{z^2}{2}}  \left|   e^{ z^2\alpha (\alpha-1) \varphi\left(\frac{z}{\sqrt{\xi}}\right)}
\right| dz\leq \frac{1}{r^{2N}\xi^{N}}\int_{-\infty}^{+\infty} z^{2N}e^{-\alpha\frac{z^2}{2}}  \left|  e^{ z^2\alpha (\alpha-1) \varphi\left(\frac{z}{\sqrt{\xi}}\right)}\right| dz.$$
To check that this last integral may be bounded by a constant independent of $\xi$, we shall simply prove that $\Re(\varphi(u))\leq0$ for any $u\in \R$. Indeed, observe that
$$
\varphi\left(u\right)=-\frac{iu}{2} \int_0^{1} (1-t)^2 (tu+i)^{\alpha-2} dt=-\frac{iu}{2} \int_0^{1} \frac{(1-t)^2}{ (t^2u^2+1)^{\frac{2-\alpha}{2}}} e^{i\theta (\alpha-2)}   dt
$$
where $\theta=\theta(t,u)\in\left[-\frac{\pi}{2},\frac{\pi}{2}\right]$ is defined by
$$\cos(\theta)=\frac{1}{\sqrt{1+t^2u^2}}\qquad \text{ and }\qquad \sin(\theta)=\frac{tu}{\sqrt{1+t^2u^2}}.$$
Since $\alpha-2<0$, we deduce that the sign of  $\sin(\theta(\alpha-2))$ is the opposite of the sign of $u$, and thus
$$\Re(\varphi(u))=\frac{u}{2} \int_0^{1} \frac{(1-t)^2}{ (t^2u^2+1)^{\frac{2-\alpha}{2}}}\sin(\theta(\alpha-2))  dt\leq0.$$
Finally, the remaining term being polynomial, we have using that $\xi>1$, 
\begin{align*}
&\int_{|z|\geq r\sqrt{\xi}} e^{-\alpha\frac{z^2}{2}  }  \left| \frac{1}{\sqrt{2\pi \alpha}}+\sum_{p=1}^{N-1} \frac{(-1)^p}{(2p)!}
\left(\frac{z}{\sqrt{\xi}}\right)^{2p} \sum_{k=1}^{2p} \big(z^{2} \alpha(\alpha-1)\big)^k B_{2p,k}  \right|dz\\
&\qquad\leq  \frac{1}{r^{2N}\xi^{N}}\int_{-\infty}^{+\infty} z^{2N}e^{-\alpha\frac{z^2}{2}  }  \left( \frac{1}{\sqrt{2\pi \alpha}}+\sum_{p=1}^{N-1} \frac{1}{(2p)!}
\left(\frac{z}{\sqrt{\xi}}\right)^{2p} \sum_{k=1}^{2p} \big(z^{2} \alpha(\alpha-1)\big)^k B_{2p,k}  \right)dz\\
&\qquad\leq  \frac{1}{r^{2N}\xi^{N}}\int_{-\infty}^{+\infty} z^{2N}e^{-\alpha\frac{z^2}{2}  }  \left( \frac{1}{\sqrt{2\pi \alpha}}+\sum_{p=1}^{N-1} \frac{1}{(2p)!}
z^{2p} \sum_{k=1}^{2p} \big(z^{2} \alpha(\alpha-1)\big)^k B_{2p,k}  \right)dz
\end{align*}
which is also finite and thus ends the proof for $\Phi_\alpha$.\\

We now study the asymptotic expansion of $\Phi_\alpha^\prime$. Since $\alpha<2$, we may differentiate under the integral in (\ref{eq:phiint}) to obtain :
\begin{multline*}
\Phi_{\alpha}^\prime(x) = \int_0^{+\infty}e^{-\sin(\frac{\pi \alpha}{2}) \frac{z^{1+\alpha}}{1+\alpha}}z \sin\left(\cos\left(\frac{\pi \alpha}{2}\right) \frac{z^{1+\alpha}}{1+\alpha} -zx\right) dz\\
=\frac{1}{2\pi i} \int_{-\infty}^{+\infty}z e^{ie^{i\frac{\pi\alpha}{2}} \frac{z^{1+\alpha}}{1+\alpha}-izx} dz =  \frac{1}{2\pi i} x^{\frac{2}{\alpha}}\int_{-\infty}^{+\infty}ze^{i x^{1+\frac{1}{\alpha}}\left(e^{i\frac{\pi\alpha}{2}} \frac{z^{1+\alpha}}{1+\alpha} -z\right)} dz. 
\end{multline*}
Using similar computations  to those for $\Phi_\alpha$, we deduce that 
$$\Phi_\alpha^\prime(x) \equi_{x\rightarrow+\infty}  \frac{1}{\sqrt{2\pi\alpha}}x^{\frac{3-\alpha}{2\alpha}} e^{- \frac{\alpha}{\alpha+1} x^{1+1/\alpha}}\sum_{p=0}^{+\infty} 
(-1)^{p+1}d_p^{(\alpha)} x^{-p(1+1/\alpha)} $$
for some coefficients $(d_p^{(\alpha)})$. The formula for these coefficients follows then by integrating the asymptotic expansion of $\left(x^{\frac{\alpha-1}{2\alpha}} e^{\frac{\alpha}{\alpha+1} x^{1+1/\alpha}}\Phi_\alpha(x)\right)^{\prime}$ and identifying the successive powers.
\end{proof}

\subsection{Asymptotics of the sequence $(c_n^{(\alpha)}, n\geq0)$}

\begin{corollary}\label{coro:c}
There exists two constants $0<\kappa_1<\kappa_2<+\infty$ such that as $n\rightarrow +\infty$ :
$$ \kappa_1\, n \leq   \left(c_{n}^{(\alpha)}\right)^{\frac{1}{n}}\leq \kappa_2\, n.$$
\end{corollary}

\begin{proof}
To get the lower bound, we simply observe that $B_{2n,2n}= \left(\frac{1}{6}\right)^{2n}$ and keep only the last term in the sum defining $c_n^{(\alpha)}$ :
$$
c_n^{(\alpha)} \geq \frac{1}{(2n)!\sqrt{\pi}}  \left(\frac{2}{\alpha}\right)^{n}  \left(\frac{1}{6}\right)^n\Gamma\left(3n+\frac{1}{2}\right) (2(\alpha-1))^{2n}.
$$
The result then follows from Stirling's asymptotics.
To get the upper bound, we first use Lemma \ref{lem:Binfty} to write
$$ c_n^{(\alpha)} \leq  \frac{1}{\sqrt{\pi}} \left(\frac{2}{\alpha}\right)^{n} \sum_{k=1}^{2n}  \frac{1}{k!}\Gamma\left(n+k+\frac{1}{2}\right) \left(\frac{\alpha-1}{2}\right)^{k}.$$
We now decompose this sum according as $k\leq n$ or $k\geq n+1$. For $k\leq n$, we have :
$$\frac{\Gamma\left(n+k+\frac{1}{2}\right)}{k!} =\left(k+\frac{1}{2}\right)\left(k+\frac{3}{2}\right)\ldots \left(k+\frac{1}{2}+n-1\right)\leq (2n)^n$$
hence  
\begin{equation}\label{sum:1}
\sum_{k=1}^{n} \left(\frac{\alpha-1}{2}\right)^{k} \frac{\Gamma\left(n+k+\frac{1}{2}\right)}{k!} \leq   \frac{\alpha-1}{3-\alpha} (2n)^n.
\end{equation}
For $k\geq n+1$, we deduce from Stirling's asymptotics that there exists a constant $\kappa>0$ such that 
\begin{align}
\notag\sum_{k=n+1}^{2n} \left(\frac{\alpha-1}{2}\right)^{k} \frac{\Gamma\left(n+k+\frac{1}{2}\right)}{k!}&\leq \kappa e^{-n}  \sum_{k=n+1}^{2n} \left(\frac{\alpha-1}{2}\right)^{k}    \frac{\left(n+k+\frac{1}{2}\right)^{n+k}}{k^{k+\frac{1}{2}}}\\
\notag&\leq \kappa e^{-n}   \left(3n+\frac{1}{2}\right)^{n-\frac{1}{2}} \sum_{k=n+1}^{2n}  \left(\frac{\alpha-1}{2}\right)^{k} \left(1+\frac{n+\frac{1}{2}}{k}\right)^{k+\frac{1}{2}}\\
\notag&\leq \kappa e^{-n}   \left(3n+\frac{1}{2}\right)^{n-\frac{1}{2}} \sum_{k=n+1}^{2n}  \left(\frac{\alpha-1}{2}\right)^{k} 2^{k+\frac{1}{2}}\\
\label{sum:2} & \leq  \kappa e^{-n}   \left(3n+\frac{1}{2}\right)^{n-\frac{1}{2}} \sqrt{2} n 
\end{align}
The result now follows by summing (\ref{sum:1}) and (\ref{sum:2}).
\end{proof}

\addcontentsline{toc}{section}{References}
%\nocite{*}
%\bibliographystyle{alpha}
%\bibliography{Biblio}

\end{document}